\documentclass{amsart}

\usepackage{amssymb, amsmath}
\usepackage{mathrsfs}
\usepackage{amscd}
\usepackage{verbatim}
\usepackage{xcolor}
\usepackage{enumerate}

\allowdisplaybreaks
\theoremstyle{plain}
\newtheorem{theorem}{Theorem}[section]
\theoremstyle{remark}

\newtheorem{example}[theorem]{Example}

\theoremstyle{plain}
\newtheorem{corollary}[theorem]{Corollary}
\newtheorem{lemma}[theorem]{Lemma}
\newtheorem{proposition}[theorem]{Proposition}
\newtheorem{definition}[theorem]{Definition}

\numberwithin{equation}{section}

\def\N{{\mathbb N}}

\def\R{{\mathbb R}}

\def\C{{\mathbb C}}

\newcommand{\F}{{\mathscr F}}

\newcommand{\Rr}{\ensuremath{\mathcal{R}}}

\def \Ell{\rm Ell}

\newcommand{\calL}{{\mathscr L}}

\newcommand{\one}{{{\bf 1}}}

\newcommand{\lb}{\langle}
\newcommand{\rb}{\rangle}

\newcommand{\Schw}{{\mathscr S}}

\newcommand{\calK}{\mathcal{K}}

\def\avint_#1{\mathchoice%
      {\mathop{\kern 0.2em\vrule width 0.6em height 0.69678ex depth -0.58065ex
              \kern -0.8em \intop}\nolimits_{\kern -0.4em#1}}%
      {\mathop{\kern 0.1em\vrule width 0.5em height 0.69678ex depth -0.60387ex
              \kern -0.6em \intop}\nolimits_{#1}}%
      {\mathop{\kern 0.1em\vrule width 0.5em height 0.69678ex depth -0.60387ex
              \kern -0.6em \intop}\nolimits_{#1}}%
      {\mathop{\kern 0.1em\vrule width 0.5em height 0.69678ex depth -0.60387ex
              \kern -0.6em \intop}\nolimits_{#1}}}

\def \R{ \mathbb{R} }
\def \N{ \mathbb{N} }
\def \T{ \mathscr{T} }
\def \I{ \mathscr{I} }
\def \C{ \mathbb{C} }

\def \F{ \mathcal{F} }

\def \M{\mathcal{M}}

\def \Ell{\rm Ell}

\begin{document}
\numberwithin{equation}{section}

\author{Chiara Gallarati}
\address{Delft Institute of Applied Mathematics\\
Delft University of Technology \\ P.O. Box 5031\\ 2600 GA Delft\\The
Netherlands} \email{C.Gallarati@tudelft.nl}

\author{Mark Veraar}
\email{M.C.Veraar@tudelft.nl}

\date\today

\title
[Evolution families and maximal regularity for systems]{Evolution families and maximal regularity for systems of parabolic equations}

\begin{abstract}
In this paper we prove maximal $L^p$-regularity for a system of parabolic PDEs, where the elliptic operator $A$ has coefficients which depend on time in a measurable way and are continuous in the space variable. The proof is based on operator-theoretic methods and one of the main ingredients in the proof is the construction of an evolution family on weighted $L^q$-spaces.
\end{abstract}



%42B25  View Publications (1980-now) Maximal functions, Littlewood-Paley theory
%42B37  View Publications (2010-now) Harmonic analysis and PDE [See also 35-XX]
%46E30  View Publications (1973-now) Spaces of measurable functions (Lp-spaces, Orlicz spaces, Köthe function spaces, Lorentz spaces, rearrangement invariant spaces, ideal spaces, etc.)
%47A56  View Publications (1980-now) Functions whose values are linear operators (operator and matrix valued functions, etc., including analytic and meromorphic ones)
%47B55  View Publications (1973-1990) Operators on ordered spaces

%42B15  View Publications (1980-now) Multipliers
%47D06  View Publications (1991-now) One-parameter semigroups and linear evolution equations [See also 34G10, 34K30]
%42B37  View Publications (2010-now) Harmonic analysis and PDE [See also 35-XX]
%42B20  View Publications (1980-now) Singular and oscillatory integrals (Calderón-Zygmund, etc.)
%35B65  View Publications (1980-now) Smoothness and regularity of solutions
%35K30  View Publications (1973-now) Initial value problems for higher-order parabolic equations
%35K90  View Publications (2000-now) Abstract parabolic equations
%34G10  View Publications (1980-now) Linear equations [See also 47D06, 47D09]

\thanks{The first author is supported by Vrije Competitie subsidy 613.001.206 and the second author by the Vidi subsidy 639.032.427 of the Netherlands Organisation for Scientific Research (NWO)}

\maketitle

\section{Introduction}

In this paper we study the evolution equation
\begin{equation}\label{prob:systemintro}
%\begin{aligned}
u'(t,x)+A(t)u(t,x)=f(t,x),\ t\in\R,\ x\in\R^{d}\\
%\end{aligned}
\end{equation}
with and without initial value condition, and where $A$ is given by
\begin{equation} %\label{eq:operator}
(A(t) u)(x) = \sum_{|\alpha|, |\beta|\leq m}  a_{\alpha\beta}(t,x)D^{\alpha} D^\beta u(t,x),
\nonumber
\end{equation}
and satisfies the Legendre-Hadamard ellipticity condition. The coefficients $a_{\alpha}:\R\times \R^d\to  \C^{N\times N}$ are assumed to be measurable in time and continuous in space.

Below in Theorems \ref{teo:mainxdep} and \ref{teo:mainxdepdiv} we derive maximal $L^p(L^q)$-regularity for \eqref{prob:systemintro} by applying a general operator-theoretic method recently discovered by the authors in \cite{GV}. There, the application to the scalar case (i.e.\ $N=1$) has already been considered.
In order to apply our method we construct the evolution family $S(t,s)$ generated by $A(t)$ in the case the coefficients are space independent.
This ``generation'' result is interesting on its own and can be found below in Theorem \ref{teoMihCondSys}.
Its proof is based on Fourier multiplier theory. Since we are dealing with systems the symbol is not explicitly known and only given as the solution to an ordinary differential equations. In order to provide estimates for the symbol, we use the implicit function theorem.
This paper gives another class of examples to which \cite{GV} is applicable. In future works we will consider applications to PDEs on bounded domains as well.

In the case the operator $A$ is time-independent an operator-theoretic characterization of maximal $L^p$-regularity has been obtained in \cite{Weis}. By this result it is enough to understand the precise behavior of the resolvents of $A$ instead of the parabolic problem itself. This approach has a wide range of applications and the development of the theory is still in progress (see \cite{DHP,KW} and references therein). In the case the coefficients are time-dependent the characterization of \cite{Weis} remains true under the additional assumption that $t\mapsto A(t)$ is continuous (see \cite{PS01}). Without continuity assumptions such a result seems to be unavailable.

A completely different approach has been developed in a series of papers in which $L^p(L^p)$-regularity results are derived where the coefficients are measurable in time and VMO (vanishing mean oscillation) in space (see the monograph \cite{krylov} and \cite{DK11} and references therein).
In the paper \cite{Krypq} a method for the case $p \geq q$ was introduced. Very recently in \cite{DK15} the full range of $p,q\in (1, \infty)$ has been considered in the case of systems as well.

The paper is organized as follows. In Section \ref{sec:preliminaries} we explain the notation and state our main results on existence and uniqueness. In Section \ref{sec:mainresult} we prove that in the case of $x$-independent coefficients $A(t)$ generates an evolution family on weighted $L^q$-spaces. In Section \ref{sec:proofs} we present the proofs of Theorems \ref{teo:mainxdep} and \ref{teo:mainxdepdiv} and we show how to deduce maximal regularity result for the initial value problem as well.

\smallskip

{\em Acknowledgment} – The authors would like to thank the anonymous referees for the careful reading and helpful comments.

\section{Assumptions and main results}\label{sec:preliminaries}

\subsection{Weights}
The main results will be stated in a weighted setting.
Details on Muckenhoupt weights can be found in \cite[Chapter 9]{GrafakosModern} and \cite[Chapter V]{SteinHA}.
A {\em weight} is a locally integrable function $w:\R^d\to (0,\infty)$.
A function $f:\R^d\to X$ is called {\em strongly measurable} if it is the a.e.\ limit of a sequence of simple functions.
For a Banach space $X$ and $p\in [1, \infty)$, $L^p(\R^d,w;X)$ is the space of
all strongly measurable functions $f:\R^d\to X$ such that
\begin{align*}
\|f\|_{L^p(\R^d,w;X)}=\Big(\int_{\R^d} \|f(x)\|^p w(x) \, dx\Big)^\frac{1}{p}<\infty.
\end{align*}

For $p\in (1, \infty)$ a weight $w$ is said to be an {\em $A_{p}$-weight} if
\begin{align*}
   [w]_{A_{p}}:=\sup_{Q} \avint_Q w(x) \, dx \Big(\avint_Q w(x)^{-\frac{1}{p-1}}\, dx \Big)^{p-1}<\infty.
\end{align*}
Here the supremum is taken over all cubes $Q\subseteq \R^d$ with axes parallel to the coordinate axes and $\avint_{Q} = \frac{1}{|Q|}\int_{Q}$. The extended real number $[w]_{A_{p}}$ is called the {\em $A_p$-constant}. A constant $C(w)$ is called {\em $A_{p}$-consistent} whenever \[[w_1]_{A_{p}}\leq [w_2]_{A_{p}} \ \Leftrightarrow \ C(w_1)\leq C(w_2).\]
One can check that $L^p(\R^d,w;X)\subseteq L^1_{\rm loc}(\R^d;X)$ for $w\in A_p$.

Let $p\in (1, \infty)$ and $w\in A_p$. For $f\in L^p(\R^d,w;X)$ we write $f\in W^{k,p}(\R^d,w;X)$ whenever for all multiindices $\alpha\in \N^d$ with $|\alpha|\leq k$, one has
$D^{\alpha} f\in L^p(\R^d,w;X)$ (in the sense of distributions). In this case we let
\[\|f\|_{W^{k, p}(\R^d,w;X)} = \sum_{|\alpha|\leq k} \|D^{\alpha} f\|_{L^p(\R^d,w;X)}.\]

\subsection{Ellipticity}
Consider an operator $A$ of the form
\[A = \sum_{|\alpha|\leq m, |\beta|\leq m} a_{\alpha\beta}D^{\alpha}D^{\beta}\]
where $a_{\alpha\beta}\in \C^{N\times N}$ are constant matrices and $D = -i(\partial_{1}, \ldots, \partial_d)$.  The \textit{principal symbol} of $A$ is defined as
\begin{equation}\label{eq:prinSymSystem}
A_{\#}(\xi):=\sum_{|\alpha| = |\beta|=m} \lb \xi^{\alpha},\ a_{\alpha\beta}\ \xi^{\beta}\rb,\ \ \ \xi\in\R^{d}.
\end{equation}
We say that $A$ is \textit{uniformly elliptic} of angle $\theta\in (0,\pi)$ if there exists a constant $\kappa\in (0,1)$ such that
\begin{equation}\label{eq:ellcondsys}
\sigma(A_{\#}(\xi))\subseteq\Sigma_{\theta}\cap \{\xi: |\xi|\geq \kappa\}, \xi\in \R^d, |\xi|=1,
\end{equation}
and there is a constant $K\geq 1$ such that $\|a_{\alpha\beta}\|\leq K$ for all $|\alpha|,|\beta|\leq m$.
In this case we write $A\in\Ell(\theta,\kappa,K)$.

We say that $A$ is {\em elliptic in the sense of Legendre--Hadamard} (see \cite{DK11, Fried64, GiaMart}) if there exists a constant $\kappa>0$ such that
\begin{equation}\label{eq:ellcondsys2}
\text{Re}(\lb x, A_{\#}(\xi) x\rb) \geq\kappa \|x\|^2,\ \ \ \xi\in\R^{d},\ |\xi|=1, x\in \C^N
\end{equation}
and there is a constant $K$ such that $a_{\alpha\beta}\leq K$ for all $|\alpha|,|\beta|\leq m$. In this case we write $A\in \Ell^{\rm LH}(\kappa, K)$. Obviously, \eqref{eq:ellcondsys2} implies \eqref{eq:ellcondsys} with $\theta = \arccos(\kappa/\tilde{K})\in (0,\pi/2)$, where $\tilde{K}$ depends only on $m$ and $K$.

\subsection{$L^p(L^q)$-theory for Systems of PDEs with time-dependent coefficients}
In order to state the main result consider the following system of PDEs
\begin{equation}\label{prob:systemintro2}
u'(t,x)+(\lambda+A(t))u(t,x)=f(t,x),\ \ \ t\in\R,\ x\in\R^{d}
\end{equation}
where $u,f:\R\times \R^d\to \C^N$ and $A$ is the following differential operator of order $2m$:
\begin{equation}\label{eq:defoperator}
(A(t) u)(x) = \sum_{|\alpha|\leq m,|\beta|\leq m} a_{\alpha\beta}(t,x)D^{\alpha} D^\beta u(x),
\end{equation}
where $a_{\alpha\beta}:\R\times \R^{d}\rightarrow \C^{N\times N}$. A function $u:\R\times\R^d\to \C^N$ is called a {\em strong solution} to \eqref{prob:systemintro2} when all the above derivatives (in distributional sense) exist in $L^1_{\rm loc}(\R\times\R^d;\C^N)$ and \eqref{prob:systemintro2} holds almost everywhere.

For $A$ of the form \eqref{eq:defoperator} and $x_0\in \R^d$ and $t_0\in \R$ let us introduce the notation:
\[A(t_0,x_0) := \sum_{|\alpha|\leq m, |\beta|\leq m} a_{\alpha\beta}(t_0,x_0) D^{\alpha}D^{\beta}.\]
for the operator with constant coefficients.

The coefficients of $A$ are only assumed to be measurable in time. More precisely, the following conditions on the coefficients are supposed to hold:
\begin{enumerate}
\item[\textbf{(C)}] Let $A$ be given by \eqref{eq:defoperator} and assume each $a_{\alpha\beta}:\R\times\R^d\to \C^{N\times N}$ is measurable.
We assume there exist $\kappa$, $K$ such that for all $t_0\in \R$ and $x_0\in \R^d$, $A(t_0, x_0)\in \Ell^{\rm LH}(\kappa, K)$.
Assume there exists an increasing function $\omega:(0,\infty)\to (0,\infty)$ with the property $\omega(\varepsilon)\rightarrow 0$ as $\varepsilon\downarrow 0$ and such that
    \begin{align*}
    \|a_{\alpha\beta}(t,x)-a_{\alpha\beta}(t,y)\|\leq \omega(|x-y|), \ \ |\alpha| = |\beta|= m, \ t\in \R, \ x,y\in\R^{d}.
\end{align*}
\end{enumerate}

The first main result is on the maximal regularity for \eqref{prob:systemintro2}.
\begin{theorem}[Non-divergence form]\label{teo:mainxdep}
Let $p,q\in (1,\infty)$, $v\in A_{p}(\R)$, $w\in A_{q}(\R^{d})$, $X_0 = L^q(\R^d,w;\C^N)$ and $X_1 = W^{2m,q}(\R^d,w;\C^N)$. Assume condition (C) holds.
Then there exists an $A_p$-$A_q$-consistent constant $\lambda_0$ such that for all $\lambda\geq\lambda_0$ and every $f\in L^{p}(\R,v;X_{0})$ there exists a unique strong solution $u\in W^{1,p}(\R,v;X_0)\cap L^p(\R,v;X_1)$ of \eqref{prob:systemintro2}. Moreover, there is an $A_p$-$A_q$-consistent constant $C$ depending on $v$, $w$, $p$, $q$, $d$, $m$, $\kappa$, $K$ and $\omega$ such that for all $\lambda\geq \lambda_0$,
\begin{equation}\label{eq:teomainxdep}
\lambda\|u\|_{L^{p}(\R,v;X_{0})}+\|u\|_{L^p(\R,v;X_1)}+\|u\|_{W^{1,p}(\R,v;X_0)}\leq C\|f\|_{L^{p}(\R,v;X_{0})}.
\end{equation}
\end{theorem}
Note that the constant $C$ does not depend on the dimension $N$. Actually, our proof allows a generalization to infinite dimensional systems, but we will not consider this here.

\medskip

A similar result holds in the case $A(t)$ is in divergence form:
\begin{equation}\label{prob:systemintrodiv}
u'(t,x)+(\lambda+A_{\rm div}(t))u(t,x)=\sum_{|\alpha|\leq m} D^{\alpha} f_{\alpha}(t,x),\ \ \ t\in\R,\ x\in\R^{d}
\end{equation}
where $u,f_{\alpha},g:\R\times \R^d\to \C^N$. Here $A_{\rm div}$ is the following differential operator of order $2m$:
\begin{equation}\label{eq:defoperatordiv}
(A_{\rm div}(t) u)(x) = \sum_{|\alpha|\leq m,|\beta|\leq m} D^{\alpha}\big(a_{\alpha\beta}(t,x)D^\beta u(x)\big),
\end{equation}
where $a_{\alpha\beta}:\R\times \R^{d}\rightarrow \C^{N\times N}$. Again we assume the same condition (C).
We say that $u\in L^1_{\rm loc}(\R\times \R^d)$ is a {\em weak solution} of \eqref{prob:systemintrodiv} if $\nabla^m u\in L^1_{\rm loc}(\R\times \R^d)$ exists in the weak sense and for all $\varphi\in C^\infty_c(\R\times\R^d)$,
\begin{align*}
\int_{\R^{d+1}}  -\big<u, \varphi'\big> + & \lambda \big< u,\varphi\big>   +  (-1)^{|\alpha|} \big< a_{\alpha\beta} D^{\beta} u, D^{\alpha}\varphi\big> \, d(t,x)
\\ & = \int_{\R^{d+1}} (-1)^{|\alpha|}  \big< f, D^{\alpha} \varphi\big> \, d(t,x),
\end{align*}
where we used the summation convention.

\begin{theorem}[Divergence form]\label{teo:mainxdepdiv}
Let $p,q\in (1,\infty)$, $v\in A_{p}(\R)$, $w\in A_{q}(\R^{d})$, $X_0 = L^q(\R^d,w;\C^N)$ and $X_{\frac{1}{2}} = W^{m,q}(\R^d,w;\C^N)$. Assume condition (C) holds for $A_{\rm div}$ as in \eqref{eq:defoperatordiv}. Then there exists an $A_p$-$A_q$-consistent constant $\lambda_0$ such that for all $\lambda\geq\lambda_0$ and every $(f_\alpha)_{|\alpha|\leq m}$ in $L^{p}(\R,v;X_{0})$ there exists a unique weak solution $u\in L^p(\R,v;X_{\frac12})$ of \eqref{prob:systemintrodiv}. Moreover, there is an $A_p$-$A_q$-consistent constant $C$ depending on $v$, $w$, $p$, $q$, $d$, $m$, $\kappa$,$K$ and $\omega$ such that
\begin{equation}\label{eq:teomainxdepdiv}
\begin{aligned}
\sum_{|\alpha|\leq m} \lambda^{1-\frac{|\alpha|}{m}} \|D^{\alpha}u\|_{L^{p}(\R,v;X_{0})} & \leq C\sum_{j=1}^d \lambda^{\frac{|\alpha|}{m}} \|f_\alpha\|_{L^{p}(\R,v;X_{0})}.
\end{aligned}
\end{equation}
\end{theorem}

\section{Generation of evolution families}\label{sec:mainresult}

In this section we will show that in the case $A(t)$ has $x$-independent coefficient it generates a strongly continuous evolution family $S(t,s)$. Recall that a function is called strongly continuous if it is continuous in the strong operator topology.
Before we turn to the proof, we recall some generalities on evolution families. For details on evolution families we refer to \cite{AT2, EN, Lun, Pazy, Schn, Ta1, Ta2, Ya} and references therein.

Assume $X_1\hookrightarrow X_0$ are Banach spaces.
%\begin{definition}
%	Let $X$ be a Banach space. A map $T: \R_+ \rightarrow \mathcal{L}(X)$ is strongly continuous if $\lim_{t\rightarrow 0}T(t)x=x$ $\forall\ x\in X$
%\end{definition}
Recall the following definition from \cite{GV}.

\begin{definition}\label{def:evolut}
Let $(A(t))_{t\in \R}$ be a family of bounded linear operators from $X_1$ into $X_0$ such that for all $x\in X_1$, $t\mapsto A(t)x$ is measurable.
A two parameter family of bounded linear operators $S(t,s)$, $s\leq t$, on a Banach space $X_{0}$ is called an \textit{evolution family for $(A(t))_{t\in \R}$} if the following conditions are satisfied:
\begin{itemize}
\item[(i)] $S(s,s)=I,\ S(t,r)S(r,s)=S(t,s)$ for $s\leq r\leq t$;
\item[(ii)] $(t,s)\rightarrow S(t,s)$ is strongly continuous for $s\leq t$.
\item[(iii)] For all $s\in \R$ and $T\in (s, \infty)$, for all $x\in X_1$, the function $u:[s,T]\to X_0$
defined by $u(t) = S(t,s)x$ is in $L^1(s,T;X_1)\cap W^{1,1}(s,T;X_0)$ and satisfies $u'(t) + A(t) S(t,s)x=0$ for almost all $t\in(s,T)$.
\item[(iv)]
For all $t\in \R$ and $T\in (-\infty, t]$ for all $x\in X_1$, the function $u:[T,t]\to X_0$ defined by $u(s) = S(t,s)x$ is in $L^{1}(T,t;X_1)\cap W^{1,1}(T,t;X_0)$ and satisfies $u'(s) = S(t,s) A(s) x$ for almost all $s\in (T, t)$.
\end{itemize}
\end{definition}
The above definition differs from the usual one from the literature, because $t\mapsto A(t)$ is only assumed to be measurable in time. Therefore, one cannot expect $S(t,s)x$ to be differentiable in the classical sense.

For a strongly measurable function $f:(a,b)\to X_{0}$ and $x\in X_0$ consider:
\begin{equation}\label{eqprobabstinvalprel}
\begin{cases}
u'(t)+A(t)u(t)=f(t),\ t\in (s,b)\\
u(s)=x.
\end{cases}
\end{equation}
If $A(t)$ generates an evolution family $S(t,s)$, then for all $x\in X_1$, $u(t) = S(t,s)x$ is a strong solution to \eqref{eqprobabstinvalprel} with $f=0$.

\subsection{On the sectoriality of the operator}

First consider the case $a_{\alpha\beta}$ is time and space independent:

\begin{equation}\label{eq:operatortimexind}
(A u)(x) = \sum_{|\alpha| = |\beta|=m}  a_{\alpha\beta}D^{\alpha} D^\beta u(x),
\end{equation}

The next result can be found in \cite[Theorem 3.1]{HHH}, where the case of $x$-dependent coefficients is considered as well.
\begin{theorem}\label{teoHHHsys}
Let $A$ be of the form \eqref{eq:operatortimexind} and assume there exist
$\theta$, $\kappa>0$ and $K>0$ such that $A \in \Ell(\theta,\kappa,K)$.
Let $1<q<\infty$ and $w\in A_q$ and let $X_0 = L^q(\R^d,w;\C^N)$.
Then there exists an $A_{q}$-consistent constant $C$ depending on the parameters $m,d,\kappa,K,q$ such that
\begin{equation}\label{eqteoHHHsys}
\|\lambda^{1-\frac{|\beta|}{m}}D^{\beta}(\lambda+A)^{-1}\|_{\calL(X_0)}\leq C,\ \ \ |\beta|\leq m,\ \lambda\in\Sigma_{\pi-\theta}.
\end{equation}
\end{theorem}
Later on the above result will be applied to the operator $A(t)$ for fixed $t\in \R$.
To prove \eqref{eqteoHHHsys} it suffices to check that for every $\lambda\in \Sigma_{\pi-\theta}$, and $|\beta|\leq m$, the symbol $\mathcal{M}:\R^d\to \C$ given by
\[\mathcal{M}(\xi) = \lambda^{1-\frac{|\beta|}{m}} \xi^{\beta} (\lambda+A_{\#}(\xi))^{-1}\]
satisfies the following type of Mihlin's condition: for every multiindex $\alpha\in \N^d$,  there is a constant $C_{\alpha}$ which only depends on
$d, \alpha, \theta,\theta_0, K, \kappa$ such that
\begin{equation}\label{eq:Mihlincond}
|\xi|^{\alpha} |D^{\alpha} \mathcal{M}(\xi)| \leq C_{\alpha}, \ \ \ \xi\in \R^d.
\end{equation}
Indeed, then the result is a consequence of the weighted version of Mihlin's multiplier theorem as in \cite[Theorem IV.3.9]{GarciaRubio}. Note  that this extends to the $\calL(\C^N)$-valued case (see \cite[Theorem 6.1.6]{BergLof} for the unweighted case). The proof of \eqref{eq:Mihlincond} follows from elementary calculus and the following lemma taken from \cite[Proposition 3.1]{DuongSimonett}. For convenience and in order to track the constants in the estimates, we present the details.

\begin{lemma}\label{Lemmaeq3.6duongsimon}
Let $A\in\Ell(\theta_0,\kappa,K)$ be of the form \eqref{eq:operatortimexind}
with $\kappa\in(0,1)$, $K>0$ and $\theta_0\in (0,\pi)$. Let $\theta\in (\theta_0,\pi)$ be fixed.
Then there is a positive constant $C=C(\kappa,\theta_0,\theta)$ such that
\begin{equation}\label{eq3.6DS}
\|(A_{\#}(\xi) + \lambda)^{-1}\|\leq C(|\xi|^{m} + |\lambda|)^{-1},\ \ (\lambda,\xi)\in \Sigma_{\pi-\theta}\backslash\{0\}\times\R^{d},
\end{equation}
where $A_{\#}$ is the principal symbol of $A$.
\end{lemma}
\begin{proof}
To start, we recall a general observation from \cite[Lemma 4.1]{AHS}. If $B\in\mathcal{L}(\C^N)$ with $\sigma(B)\subseteq\{z:|z|\geq r\}$ for some $r>0$, then one has
\begin{equation}\label{eqLemma4.1AHS}
\|B^{-1}\|\leq \|B\|^{n}r^{-n-1}, \ \ n\geq 0.
\end{equation}
Indeed, to show this it suffices to consider the case $r=1$.
Since $\|B^{*}B\|=\|B\|^{2}$, it is sufficient to consider self-adjoint $B$. Let $\lambda_{\rm min}, \lambda_{\rm max}\geq 1$ be the smallest and largest eigenvalue of $B$ respectively. The observation follows from
\[
\|B^{-1}\|=\frac{1}{\lambda_{\rm min}}\leq 1 \leq (\lambda_{\rm max})^{n}=\|B\|^{n}.
\]	

We then claim that with $\varepsilon=\sqrt{\tfrac{{1-b}}{{2}}}$ and $b=|\cos(\theta-\theta_0)|$,
\begin{equation}\label{eqineqtoverify}
|\lambda + \mu|\geq \varepsilon (|\lambda|+|\mu|), \ \forall
\ \xi\in \R^d, \ \lambda\in \Sigma_{\pi-\theta}\backslash\{0\}, \ \mu\in \sigma(A_{\#}(\xi)).
\end{equation}
To prove the claim, write $\mu=|\mu|e^{i\varphi}$ with $|\varphi|\leq\theta_0$ and
$\lambda=|\lambda|e^{i\psi}$ with $|\psi|\leq\pi-\theta$.
Clearly, $|\psi-\varphi|\leq\pi-(\theta-\theta_0)$, from which we see $\cos(\psi-\varphi)\geq-b$.
Therefore, the claim follows from the elementary estimates
\begin{align*}
| \lambda + \mu|^{2}&= |\lambda|^2 + |\mu|^2 + 2 \text{Re}(\lambda \overline{\mu})
= |\lambda|^{2} + |\mu|^2+2|\lambda|\,|\mu| \cos(\psi-\varphi)
\\ &\geq
|\lambda|^{2}+|\mu|^2 - 2 b|\lambda|\,|\mu| \geq (1-b)(|\lambda|^2+|\mu|^2)\geq \varepsilon (|\lambda|+|\mu|)^2.
\end{align*}

The assumptions on $A_{\#}$ and homogeneity yield
\begin{equation}\label{eqperdiseq}
\sigma(A_{\#}(\xi))\subseteq \Sigma_{\theta_0}\cap\{z:|z|\geq\kappa|\xi|^{2m}\},\ \ \xi\in\R^{d}.
\end{equation}
This implies that for all $(\lambda,\xi)\in \Sigma_{\pi-\theta}\backslash\{0\}\times\R^{d}$ with $|\lambda|+|\xi|^{2m}=1$,
\begin{equation}\label{eq:lambdaplusA}
\sigma(\lambda+A_{\#}(\xi))\subseteq\{z:|z|\geq \varepsilon \kappa\}.
\end{equation}
Indeed, if $\mu\in A_{\#}(\xi)$, then from \eqref{eqperdiseq} and \eqref{eqineqtoverify} we see that
\begin{equation}
|\lambda+\mu|\geq \varepsilon (|\lambda|+|\mu|) \geq \varepsilon (|\lambda|+\kappa|\xi|^{2m})
\geq\kappa \varepsilon \bigg(\frac{|\lambda|}{\kappa}+|\xi|^{2m}\bigg)\geq \varepsilon \kappa
\nonumber
\end{equation}
From  \eqref{eqLemma4.1AHS} and \eqref{eq:lambdaplusA} we can conclude $\|(\lambda+A_{\#}(\xi))^{-1}\|\leq (\varepsilon \kappa)^{-1}$,
with $(\lambda,\xi)\in \Sigma_{\pi-\theta}\backslash\{0\}\times\R^{d}$ and $|\lambda|+|\xi|^{2m}=1$. By homogeneity we obtain \eqref{eq3.6DS} with $C = (\varepsilon \kappa)^{-1}$.
\end{proof}

As a consequence we obtain the following:
\begin{corollary}\label{coro:HHHsys}
Let $\lambda_0 >0$. Under the conditions of Theorem \ref{teoHHHsys}, the operator $A$ on $X_0$ with domain $X_1 = W^{2m,q}(\R^d,w;\C^N)$ is closed and for every $\lambda\geq\lambda_0$,
\[
c\|u\|_{X_1}\leq \|(\lambda+A)u\|_{X_0}\leq (2K+\lambda)\|u\|_{X_1},
\]
where $c^{-1}$ is $A_{q}$-consistent and only depends on $m,d,\theta_0,\theta,\kappa,K,q$.
\end{corollary}

\subsection{Generation theorem}

Consider $A$ with time-dependent coefficients:
\begin{equation}\label{eq:defoperatortimeind}
(A(t) u)(x) = \sum_{|\alpha| = |\beta|= m} a_{\alpha\beta}(t)D^{\alpha} D^\beta u(x),
\end{equation}
with $A(t)\in \Ell^{\rm LH}(\kappa,K)$ for some $\kappa, K>0$ independent of $t\in \R$.
It follows from Theorem \ref{teoHHHsys} that $A(t)$ is a sectorial operator and by Corollary \ref{coro:HHHsys} the graph norm of
$\|u\|_{D(A(t))}$ is equivalent to the norm $\|u\|_{W^{2m,q}(\R^d,w;\C^N)}$ with uniform estimates and constants which only depend on  $w, q, d, \kappa,K,m$.
%for all $t\in \R$,
%\[\|u\|_{W^{2m,q}(\R^d,w;\C^N)} \eqsim \|A(t) u\|_{L^q(\R^d,w;\C^N)} + \|u\|_{L^q(\R^d,w;\C^N)}, \ \ u\in W^{2m,q}(\R^d,w;\C^N),\]
%with constants depending on $w, q, d, \kappa,m$.

The main result of this section is that $(A(t))_{t\in \R}$ generates a strongly continuous evolution family $(S(t,s))_{-\infty<s\leq t<\infty}$ on $L^q(\R^d,w;\C^N)$ for all $q\in (1, \infty)$ and $w\in A_q$. Recall that $u(t) = S(t,s) g$ if and only if
\begin{equation}\label{eq:Cauchyproblem}
\begin{aligned}
  u'(t) + A(t) u(t) &= 0, \ \ \text{for almost all} \ t\in (s, \infty),
  \\ u(s) &= g.
\end{aligned}
\end{equation}

\begin{theorem}[Generation of the evolution family]\label{teoMihCondSys}
Let $q\in (1, \infty)$, $w\in A_q$ and set $X_0 = L^q(\R^d,w;\C^N)$ and $X_1 = W^{2m,q}(\R^d,w;\C^N)$. Assume that there exists $\kappa, K>0$ such that for each $t\in \R$, $A(t)\in\Ell^{\rm LH}(\kappa,K)$. Then, the operator family $(A(t))_{t\in \R}$ with $D(A(t)) = X_1$ generates a unique strongly continuous evolution family $(S(t,s))_{s\leq t}$ on $X_0$. Moreover, the evolution family satisfies the following properties.
\begin{enumerate}
\item\label{eq:Scont} $(t,s)\mapsto S(t,s)\in \calL(X_0)$ is continuous on $\{(t,s):s<t\}$.
\item\label{eq:SDalphaest} for all $\alpha\in \N^d$ there is a constant $C$ such that
\[\|D^{\alpha} S(t,s)\|_{\calL(X_0)} \leq C|t-s|^{-|\alpha|/(2m)}, \ \  s<t,\]
where $C$ only depends on $q, d,\kappa,K,m$ and on $w$ in an $A_q$-consistent way.
\item\label{eq:SDalphacomm} for all $k\in \N$, and multiindices $\alpha$ with $|\alpha|\leq k$,
\[D^{\alpha} S(t,s) u = S(t,s) D^{\alpha} u, \ \ \text{for all} \  u\in W^{k,q}(\R^d,w;\C^N), s<t.\]
\item The following weak derivatives exists for almost every $s<t$,
\begin{align}
D_t S(t,s) & = - A(t) S(t,s) \ \text{on} \ \calL(X_0)  \label{eq:propevol1}
\\ D_s S(t,s) & = S(t,s)A(s) \ \text{on} \ \calL(X_1, X_0). \label{eq:propevol2}
\end{align}
\end{enumerate}
\end{theorem}

As far as we know the existence and uniqueness of the evolution family was unknown even in the case $w=1$ and $q=2$. The main difficulty in obtaining the evolution family is that the operators $A(t)$ and $A(s)$ do not commute in general. If they were commuting, then a more explicit formula for the evolution family exists (see \cite[Example 4.4]{GV}).
\begin{example}
An example where the operators are not commuting can already be given in the case $m=d=1$, $N=2$ by taking $A(t) = a(t)D_1^2$, where
\[a(t) = \left(
           \begin{array}{cc}
             1 & \one_{(0,\infty)}(t) \\
             \one_{(-\infty,0)}(t) & 1 \\
           \end{array}
         \right)
\]
One can check that $a(1)$ and $a(-1)$ are not commuting. Furthermore, one can check that the ellipticity condition \eqref{eq:ellcondsys2} holds.
\end{example}
%
%\begin{example}
%The following example shows that the operators $A(t)$ and $A(s)$ do not commute in general.
%Take $d=1$, $N=2$ and $m=1$. The operator $A(t)$ is then of the form $A(t)=a_{11}(t)D_1^2$, where $a_{11}(t)$ is a $2\times 2$-matrix. Choose $a_{11}(t)$ such that in the Fourier domain
%\begin{align*}
%&(a_{11})_{11}(t)=1=(a_{11})_{22}(t)\ \forall t\in\R,\\
%&(a_{11})_{12}(t)=1/2\ {\rm and}\ (a_{11})_{21}(t)=0\ {\rm for}\ t\leq 0,\\
%&(a_{11})_{12}(t)=0\ {\rm and}\ (a_{11})_{21}(t)=1/2\ {\rm for}\ t>0,
%\end{align*}
% that is, we have the matrices
%	$\bigg(\begin{matrix}
%		1 & 1/2\\
%		0 & 1
%	\end{matrix}\bigg)$
%	and
%	$\bigg(\begin{matrix}
%	 	1 & 0 \\
%	 	1/2 & 1
%	 \end{matrix}\bigg)$, for $t\leq 0$ and $t>0$ respectively. Observe that they do not commute. Moreover, one can easily check that with this choice for the matrix, $A$ satisfies the ellipticity condition \eqref{eq:ellcondsys2}. Therefore, if we take $t\leq 0$ and $s>0$, we showed that $A(t)$ and $A(s)$ do not commute in the Fourier domain and therefore they do not commute.
%\end{example}

In the proof below we use Fourier multiplier theory. It turns out that the symbol is only given implicitly as the solution to a system of differential equation. In order to check the conditions of Mihlin's theorem we apply the implicit function theorem.

We will need the following simple lemmas in the proof.
\begin{lemma}[Gronwall for weak derivatives]\label{lem:Gron}
Let $-\infty<s<T<\infty$, $f\in L^1(s,T)$, $a\in L^\infty(s,T)$ and $x\in \R$. Assume $u\in W^{1,1}(s,T)\cap C([s,T])$ satisfies
\[u'(t) \leq a(t) u(t) + f(t), \ \ \ \text{for almost all} \ t\in (s,T),\]
and $u(s) = x$. Let $\sigma(t,r) = e^{a_{tr}}$ and $a_{tr}= \int_r^t a(\tau) \, d\tau$ for $s\leq r<t\leq T$. Then
\[u(t)\leq \sigma(t,s) x + \int_s^t \sigma(t,r) f(r) \, dr, \ \ t\in (s,T).\]
\end{lemma}
This follows if one integrates the estimate
$\frac{d}{dr}\big[u(r) e^{-a_{rs}}\big] \leq e^{-a_{rs}} f(r)$ over $(s,t)$.
\begin{comment}
\begin{proof}
By the product rule for weak derivatives we find
\[\frac{d}{dr}\big[u(r) e^{-a_{rs}}\big] = e^{-a_{rs}} [u'(r) - a(r) u(r)]\leq e^{-a_{rs}} f(r).\]
for almost all $r\in (s,T)$.
Integrating over $r\in [s,t]$ we obtain
\[u(t) e^{-a_{ts}}\leq x + \int_s^t e^{-a_{rs}} f(r)\, dr.\]
from which the result follows.
\end{proof}
\end{comment}
The following existence and uniqueness result will be needed.

\begin{lemma}\label{lem:fixedpointW1}
Let $X$ be a Banach space and $p\in [1, \infty)$. Let $Q:\R\times X\to X$ be measurable and assume there are constants $K_1$ and $K_2$ such that for all $t\in \R$ and $x,y\in X$, $\|Q(t,x) - Q(t,y)\|\leq K_1\|x-y\|$ and $\|Q(t,x)\|\leq K_2(1+\|x\|)$. Let $u_0\in X$ and $f\in L^p(\R;X)$. Fix $s\in \R$. Then there is a unique function $u\in C([s,\infty);X)$ such that
\begin{align*}
u(t) - u_0 &= \int_s^t Q(s,u(s)) + f(s)\, ds, \ \ \ t\geq s.
\end{align*}
Moreover, with $\lambda = K_1+1$, there is a $C\geq 0$ independent of $f$ and $u_0$ such that
\[ \sup_{t\geq s} e^{-\lambda (t-s)}\|u(t)\| \leq  C\big(1+\|u_0\|+\|f\|_{L^p(s,\infty;X)}\big).\]
\end{lemma}
This is immediate from the Banach fixed point theorem applied on the space
$E_{\lambda}$ of continuous functions $u:[s, \infty)\to X$ for which
\[\|u\|_{E_{\lambda}} := \sup_{t\geq s} e^{-\lambda(t-s)} \|u(t)\| <\infty.\]

Now we can proof the generation result.
\begin{proof}[Proof of Theorem \ref{teoMihCondSys}]
The proof is divided in several steps. Let $B = \C^{N\times N}$ with the operator norm and let $\R^d_* = \R^d\setminus \{0\}$.

\medskip

{\em Step 1:} Fix $s\in \R$. Let $I$ denote the $N\times N$ identity matrix.  We will first construct the operators $S(t,s)$ and check that \eqref{eq:SDalphaest} holds for $|\alpha|=0$.
For this we show that the function $v$ given by
\begin{equation}\label{eq:CauchyproblemFourier1}
\begin{aligned}
  v_t(t,\xi) + A_{\#}(t,\xi) v(t,\xi) & = 0,
 \\ v(s,\xi) &= I,
\end{aligned}
\end{equation}
is an $L^q(\R^d,w;\C^{N})$-Fourier multiplier by applying a Mihlin multiplier theorem for weighted $L^q$-spaces (see \cite[Theorem IV.3.9]{GarciaRubio} for the case $N=1$). The solution $u$ to \eqref{eq:Cauchyproblem} is then given by
\[u(t) = S(t,s) g = \F^{-1}(v(t,\cdot) \hat{g}),\]
where $\hat{g}$ denotes the Fourier transform of $g$. Note that by Lemma \ref{lem:fixedpointW1} for each $\xi\in \R^d_*$ there exists a unique solution $v(\cdot, \xi)\in C([s,\infty);B)$ of \eqref{eq:CauchyproblemFourier1}.
Conversely, if $S(t,s)$ is an evolution family for $A(t)$, then by applying the Fourier transform, one sees that $\F(S(t,s))$ has to satisfy \eqref{eq:CauchyproblemFourier1} for almost all $t>s$. This yields the uniqueness of the evolution family.

To check the conditions of the multiplier theorem it suffices to prove the following claim: It holds that $v(t, \cdot)\in C^\infty(\R^d_*;B)$ and for all multiindices $\gamma\in \N^d$, and $j\geq 0$,
\begin{equation}\label{eq:Mihlin}
\|D^{\gamma} v(t,\xi)\|_{B} \leq C |\xi|^{-|\gamma|}, \ \ \xi\in \R^d_*,
\end{equation}
where $C$ only depends $\gamma$, $d$, $m$, $\kappa$ and $K$.
The estimate \eqref{eq:Mihlin} will be proved by induction on the length of $\gamma$ by using the implicit function theorem.

{\em Step 2:}  As a preliminary result we first prove an estimate for the problem
\begin{equation}\label{eq:CauchyproblemFourier3}
\begin{aligned}
  v_t(t,\xi) + A_{\#}(t,\xi) v(t,\xi) & = f(t,\xi),
 \\  v(s,\xi) &= M,
\end{aligned}
\end{equation}
where $f:(s,\infty)\times\R^d_*\to B$ is measurable and for each $\xi\in \R^d_*$, $f(\cdot, \xi)\in L^2(s,\infty;B)$ and $M\in B$.
Note that the existence and uniqueness of a solution $v(\cdot, \xi) \in W^{1,2}(s,T;B)\cap C([s,T];B)$ for fixed $\xi\neq 0$ and $T>s$ follows from Lemma \ref{lem:fixedpointW1}. Moreover, since $v(\cdot, \xi)$ is obtained from a sequential limiting procedure in the Banach fixed point theorem, the function $v$ is measurable on $[s, T] \times \R^d_*$. Choosing $T$ arbitrary large, it follows that there is a unique measurable $v:[s, \infty)\times \R^d_*\to B$ for which the restriction to $[s,T]$ satisfies $v(\cdot, \xi) \in W^{1,2}(s,T;B)\cap C([s,T];B)$ and is a solution to \eqref{eq:CauchyproblemFourier3}.

Fix $\xi\in \R^d_*$, $\varepsilon\in (0,\kappa)$ and $x\in \R^N$. From the ellipticity condition \eqref{eq:ellcondsys2} and \eqref{eq:CauchyproblemFourier3} we infer that
\begin{align*}
 \frac12D_t |v(t,\xi)x|^2 & = -\text{Re}\big(\lb v(t,\xi)x,A_{\#}(t,\xi) v(t,\xi)x\rb\big) + \text{Re} \big(\lb v(t,\xi)x, f(t,\xi)x\rb\big)
 \\ & \leq -\kappa|\xi|^{2m} |v(t,\xi)x|^2 + |v(t,\xi)x| \, |f(t,\xi)x|
 \\ & \leq (\varepsilon -\kappa)|\xi|^{2m} |v(t,\xi)x|^2 + \frac{1}{4\varepsilon} |\xi|^{-2m} |f(t,\xi)x|^2,
\end{align*}
where we used $2ab \leq a^2+b^2$ on the last line.
Thus Lemma \ref{lem:Gron} yields:
\begin{equation}%\label{eq:fest}
|v(t,\xi)x|^2 \leq e^{2(\varepsilon-\kappa) |\xi|^{2m} (t-s)}| Mx|^2 + \frac{1}{2\varepsilon} \int_s^t e^{2(\varepsilon-\kappa) |\xi|^{2m} (t-r)} |\xi|^{-2m} |f(r,\xi)x|^2 \, dr.
\nonumber
\end{equation}
Taking the supremum over all $|x|\leq 1$, we find that
\begin{equation}\label{eq:estvnorm}
\|v(t,\xi)\|^2 \leq  e^{2(\varepsilon-\kappa) |\xi|^{2m} (t-s)} \|M\|^2 + \frac{1}{2\varepsilon} \int_s^t e^{2(\varepsilon-\kappa) |\xi|^{2m} (t-r)} |\xi|^{-2m} \|f(r,\xi)\|^2 \, dr.
\end{equation}
Note that if $f=0$, then the second term vanishes and we can take $\varepsilon=0$ in \eqref{eq:estvnorm}.
In this case $\|v(t,\xi)\| \leq e^{(\varepsilon-\kappa) |\xi|^{2m} (t-s)} \leq 1$ and hence \eqref{eq:Mihlin} holds for $|\gamma|=0$. Also note that if $v_j$ is the solution to \eqref{eq:CauchyproblemFourier3} with $(M,f)$ replaced by $(M_j,f_j)$ for $j=1, 2$, then by the previous estimates also
\[\|v_1(t,\xi) - v_2(t,\xi)\|^2 \leq \|M_1 - M_2\|^2 + \frac{1}{2\varepsilon} |\xi|^{-2m} \|f_1(\cdot,\xi) - f_2(\cdot,\xi)\|^2_{L^2((s,\infty);B)}.\]
Consequently, since $D_t v_1 - D_t v_2 = -A_{\#}(t,\xi)(v_1-v_2) + (f_1-f_2)$
we deduce that
\begin{align*}
\|&v_1(\cdot,\xi) - v_2(\cdot, \xi)\|_{W^{1,2}(s,T;B)} \\ & \leq C(1+|\xi|^{2m}) \|M_1 - M_2\| + C\sum_{j=-1}^1 |\xi|^{jm}
\|f_1(\cdot,\xi) - f_2(\cdot,\xi)\|_{L^2((s,\infty);B)},
\end{align*}
where $C$ does not depend on $\xi\in \R^d_*$.
Thus the solution depends in a Lipschitz continuous way on the data.

\medskip

{\em Step 3:} Fix $T>0$. Define $\Psi:\R^{d}_*\times W^{1,2}(s,T;B)\rightarrow L^{2}(s,T;B)\times B$ by
\begin{equation}
(\Psi(\xi)v)(t):=(v'(t) + A_{\#}(t,\xi) v(t),v(s)).
\nonumber
\end{equation}
Clearly, $v$ is a solution to \eqref{eq:CauchyproblemFourier3} if and only if  $\Psi(\xi)v(t)=(f,M)$.
Therefore, by the previous step for each $\xi\neq 0$, $\Psi(\xi)$ is an homeomorphism and $\Psi(\xi)^{-1}(f,M) = v(t,\xi)$.

For fixed $M\in B$ and $f\in C^\infty(\R^d_*;L^2(\R;B))$, let
\[\Phi^{f,M}:\R^{d}_*\times W^{1,2}(s,T;B)\rightarrow L^{2}(s,T;B)\times B\]
be given by
\begin{equation}
\Phi^{f,M}(\xi,v):=(\Psi(\xi)v)-(f(\xi),M).
\nonumber
\end{equation}
Now for fixed $\overline{\xi}\in \R^d_*$, $\Phi^{f,M}(\overline{\xi},v)=0$ holds if and only if $v$ is a solution to \eqref{eq:CauchyproblemFourier3}. Therefore, by the previous step there exists a unique $(\overline{\xi},\overline{v})\in\R^d_*\times W^{1,2}(s,t;B)$ such that $\Phi^{f,M}(\bar{\xi},\bar{v})=0$. The Fr\'echet derivative with respect to the second coordinate satisfies
\begin{equation}\label{eq:D2Phi}
D_2 \Phi^{f,M}(\bar{\xi},\bar{v}) v =\big(v'(t) + A_{\#}(t,\xi) v(t),v(s)\big) =  (\Psi(\xi) v) (t,\xi).
\end{equation}
Thus, also $D_2 \Phi^{f,M}(\bar{\xi},\bar{v})$ is an homeomorphism. Moreover, since $A_{\#}(t,\cdot)$ and $f$ are $C^\infty$ on $\R^d_*$, it follows that for every $v\in W^{1,2}(s,T;B)$,  $\xi\mapsto \Phi^{f,M}(\xi, v)$ is $C^\infty$ on $\R^d_*$.
Now by the implicit function theorem (see \cite[Theorem 10.2.1]{Dieu69}) there exists a unique continuous mapping $\zeta:\R^d_*\to W^{1,2}(s,T;B)$ such that $\zeta(\xi)=v(\cdot,\xi)$, $(\xi,\zeta(\xi))\in\R^{d}_*\times W^{1,2}(s,T;B)$ and $\Phi^{f,M}(\xi, \zeta(\xi)) = 0$ for every $\xi\in \R^{d}_*$. From this we obtain that the unique solution of \eqref{eq:CauchyproblemFourier3} can be expressed by
\begin{equation}
v(\cdot,\xi):=\zeta(\xi)=\Psi(\xi)^{-1}(f,M).
\nonumber
\end{equation}
Moreover, by the implicit function theorem $\zeta$ is $C^\infty$ on $\R^d_*$ as an $W^{1,2}(s,T;B)$-valued function and
\begin{equation}
\begin{aligned}
D_{\xi_j} \zeta(\xi)  & = -(\Psi(\xi))^{-1} \circ \Psi^{f,M}_{\xi_j}(\xi,\zeta(\xi)))
 =\Psi(\xi)^{-1}\Big(\tilde{f},0\Big),
\nonumber
\end{aligned}
\end{equation}
where $\tilde{f}(t,\xi) = -D_{\xi_j} A_{\#}(t,\xi)y(\xi)+D_{\xi_j}f(t,\xi)$ and where we applied \eqref{eq:D2Phi}.
This means that $D_{\xi_j} \zeta(\xi)$ is a solution to \eqref{eq:CauchyproblemFourier3} with $M = 0$ and $f$ replaced by $\tilde{f}(t,\xi)$ and that by \eqref{eq:estvnorm} the following estimate holds
\begin{equation}\label{eq:estvnorm2}
\|D_{\xi_j} \zeta(\xi)(t)\|^2 \leq  \frac{1}{2\varepsilon} \int_s^t e^{2(\varepsilon-\kappa) |\xi|^{2m} (t-r)} |\xi|^{-2m} \|\tilde{f}(r,\xi)\|^2 \, dr.
\end{equation}

\medskip

{\em Step 4.}
We are now in position to do the induction step. Assume that $\forall\ |\gamma|\leq n$ the problem
\begin{equation}\label{eq:CauchyproblemFourierdifff}
\begin{aligned}
  v_{\gamma}'(t,\xi) + A_{\#}(t,\xi) v_{\gamma}(t,\xi) & = f(t,\xi)
 \\  v_{\gamma}(s,\xi) &= M,
\end{aligned}
\end{equation}
has a unique solution given by $v_{\gamma}(t,\xi) = D^{\gamma}v(t,\xi)$,
where $M = 0$ if $|\gamma|\geq 1$, $M = I$ if $|\gamma|=0$ and $f$ is given by
\[f(t,\xi) = - \sum_{\substack{\eta_{1}+\eta_{2}=\gamma \\ \eta_{2}\neq\gamma}}c_{\eta_{1},\eta_{2}}D^{\eta_{1}}A_{\#}(t,\xi)D^{\eta_{2}}v(t,\xi)\]
and assume that $\forall\ |\gamma|\leq n$,
\begin{equation}\label{eq:estIV}
\|D^{\gamma}v(t,\xi)\|\leq C_{\gamma} |\xi|^{-|\gamma|}.
\end{equation}
Of course these assertions hold in the case $|\gamma| = 0$, by Step 2.

Fix $|\gamma|=n+1$ and write $\gamma=\tilde{\gamma}+\beta$, with $|\tilde{\gamma}|=n$, $|\beta|=1$. By Step 3, the function $w = D^{\beta} v_{\gamma}$ satisfies
\begin{equation}\label{eq:CauchyproblemFourierdiff2}
\begin{aligned}
  w'(t,\xi) + A_{\#}(t,\xi) w(t,\xi) & = \tilde{f}(t,\xi)
 \\  w(s,\xi) &= 0,
\end{aligned}
\end{equation}
and for suitable $\tilde{c}_{\eta_{1},\eta_{2}}$,
\[\tilde{f}(t,\xi) = - \sum_{\substack{\eta_{1}+\eta_{2}=\gamma \\ \eta_{2}\neq  \gamma}}\tilde{c}_{\eta_{1},\eta_{2}}D^{\eta_{1}} A_{\#}(t,\xi)D^{\eta_{2}}v(t,\xi)\]
Moreover, by \eqref{eq:estvnorm2}, the fact that $\xi\mapsto D^{\eta_1} A_{\#}(t,\xi)$ is a $(2m-|\eta_1|)$-homogenous polynomial, $|\eta_1| + |\eta_2| = n+1$, and \eqref{eq:estIV} we find
\begin{align*}
\|D^{\gamma}v(t,\xi)\|^2&\leq \frac{1}{2\varepsilon} \int_s^t e^{2(\varepsilon-\kappa) |\xi|^{2m} (t-r)} |\xi|^{-2m} \|\tilde{f}(r,\xi)\|^2 \, dr
\\ & \leq \frac{1}{2\varepsilon} \sum_{\substack{{\eta_{1}+\eta_{2}=\gamma} \\ {\eta_{2}\neq\gamma}}} \tilde{c}_{\eta_1, \eta_2} \int_s^t e^{2(\varepsilon-\kappa) |\xi|^{2m} (t-r)} |\xi|^{2m-2|\eta_1|}  \|D^{\eta_2} v(r,\xi)\|^2 \, dr
\\ & \leq \tilde{C}_{\gamma} |\xi|^{-2|\gamma|}\int_s^t e^{2(\varepsilon-\kappa) |\xi|^{2m} (t-r)} |\xi|^{2m}  \, dr \leq C_{\gamma} |\xi|^{-2|\gamma|}.
\end{align*}
This completes the induction step and hence \eqref{eq:Mihlin} follows.

\medskip

{\em  Step 5:} To prove \eqref{eq:SDalphaest} for general $\alpha$, fix $k\geq 0$. Since $\|v(t)\| \leq e^{2(\varepsilon-\kappa) |\xi|^{2m} (t-s)}$, there is a constant $C$ such that
\[\|v(t)\| \leq C |\xi|^{-2m k} |t-s|^{-k} , \ t\geq s.\]
Now if we replace the induction hypothesis \eqref{eq:estIV} by
\begin{equation}\label{eq:estIV2}
\|D^{\gamma}v(t,\xi)\|\leq C_{\gamma} |\xi|^{-|\gamma|-2mk} |t-s|^{-k} , \ \ s<t, \xi\neq 0
\end{equation}
for all $|\gamma|\leq n$, then for $\gamma =\tilde{\gamma} + \beta$ with $|\tilde{\gamma}| = n$ and $|\beta|=1$, we find
\begin{align*}
\|D^{\gamma}v(t,\xi)\|^2&\leq \frac{1}{2\varepsilon} \int_s^t e^{2(\varepsilon-\kappa) |\xi|^{2m} (t-r)} |\xi|^{-2m} \|\tilde{f}(r,\xi)\|^2 \, dr
\\ & \leq \frac{1}{2\varepsilon} \sum_{\substack{\eta_{1}+\eta_{2}=\gamma \\ \eta_{2}\neq\gamma}} \tilde{c}_{\eta_1, \eta_2} \int_s^t e^{2(\varepsilon-\kappa) |\xi|^{2m} (t-r)} |\xi|^{2m-2|\eta_1|}  \|D^{\eta_2} v(r,\xi)\|^2 \, dr
\\ & \leq \tilde{C}_{\gamma} |\xi|^{-2|\gamma|}\int_s^t e^{2(\varepsilon-\kappa) |\xi|^{2m} (t-r)} |\xi|^{2m}  |\xi|^{-4mk} (r-s)^{-2k}   \, dr
\\ & \leq \tilde{C}_{\gamma} |\xi|^{-2|\gamma|-4mk} (t-s)^{-2k} \int_s^t e^{2(\varepsilon-\kappa) |\xi|^{2m} (t-r)} |\xi|^{2m}  \, dr
\\ & \leq C_{\tilde{\gamma}} |\xi|^{-2|\gamma|-4mk} (t-s)^{-2k}.
\end{align*}
Hence by induction, \eqref{eq:estIV2} holds for all integers $n\geq 0$.

By \eqref{eq:estIV2}, $w(t,\xi) = (-i\xi)^{\alpha} v(t,\xi)$ satisfies the conditions of the Mihlin multiplier theorem, with constant $\lesssim (t-s)^{-|\alpha|/(2m)}$ and therefore we find that
\[\|D^{\alpha} S(t,s)\|_{X_0}\leq C (t-s)^{-|\alpha|/(2m)}\]
which proves \eqref{eq:SDalphaest}. The identity in \eqref{eq:SDalphacomm} is a direct consequence of the fact that $v(t,\xi) (-i\xi)^{\alpha} =(-i\xi)^{\alpha}v(t,\xi)$ .

{\em  Step 6:} Next we prove that $S(t,s)$ is a strongly continuous evolution family for $A(t)$, i.e.\ that it satisfies Definition \ref{def:evolut}.
The identities $S(t,t) = I$ and $S(t,s)S(s,r) = S(t,r)$ are clear from the definition of $v$ and Lemma \ref{lem:fixedpointW1}. To prove strong continuity of the evolution family, consider $(t,s)\mapsto S(t,s)g = \F^{-1} (v_s(t)\hat{g})$ for $g\in X_1$, where $v_s$ is the solution to \eqref{eq:CauchyproblemFourier1}. Setting $f(r) = -A(r) S(r,s) g$ it follows from \eqref{eq:SDalphacomm} that for all $r\geq s$, $\|f(r)\|_{X_0}\leq C\|g\|_{X_1}$. Moreover,
\begin{align*}
S(t,s)g - g  & = \F^{-1}(v_s(t,\cdot)\hat{g} - \hat{g})   = \F^{-1} \Big(\int_s^t \hat{f}(r) \,dr \Big) = \int_s^t f(r) \,dr
\end{align*}
in $\Schw'(\R^d;\C^N)$ and hence in $X_0$. This proves Definition \ref{def:evolut} (iii). Moreover, we find
\[\|S(t,s)g - g\|_{X_0}\leq (t-s) \sup_{r\in [s, t]}\|f(r)\|_{X_0}\leq C(t-s)\|g\|_{X_1}\]
which implies the continuity of $(t,s)\mapsto S(t,s)g$ for $g\in X_1$. The general case follows by approximation and the uniform boundedness of $S(t,s)$. It remains to prove Definition \ref{def:evolut} (iv) and this will be done in the next step.

{\em Step 7:} To prove \eqref{eq:propevol1} fix $r\in (s, t)$. Note that by \eqref{eq:SDalphaest}, $f = S(r,s)g\in W^{\ell, q}(\R^d,w;\C^N)$ for any $\ell\in \N$. Therefore, it follows from the previous step and \eqref{eq:SDalphacomm} that
\begin{equation}\label{eq:Sintder}
\begin{aligned}
S(t,s) g - S(r,s) g & = S(t,r) f  - f  = -\int_r^t A(\tau) S(\tau,r) f \, d\tau \\ &= -\sum_{|\alpha| = |\beta|=m} \int_r^t a_{\alpha, \beta} D^{\alpha} S(\tau,r) D^{\beta} f \, d\tau
\end{aligned}
\end{equation}
and since by \eqref{eq:SDalphaest}, $\|D^{\alpha} S(\tau,r)\|\leq C(\tau-r)^{-1/2}$ for $|\alpha| = m$ we find that
\begin{align*}
\|S(t,s) g - S(r,s) g\|_{X_0}
& \leq C \int_r^t (\tau - r)^{-1/2} (r-s)^{-1/2} \, d\tau \|g\|_{X_0}
\\ & \leq C (t-r)^{1/2} (r-s)^{-1/2} \|g\|_{X_0}.
\end{align*}
This implies that $t\mapsto S(t,s)\in \calL(X_0)$ is H\"older continuous on $[s+\varepsilon, \infty)$ for any $\varepsilon>0$.
Moreover, since $A$ is strongly measurable also $\tau\mapsto A(\tau) S(\tau,s)$ is a strongly measurable function. By \eqref{eq:SDalphaest}, $\|A(\tau) S(\tau,s)\|\leq C(\tau-s)^{-1}$ and hence it is locally integrable on $[s, \infty)$ as an $\calL(X_0)$-valued function. Therefore, \eqref{eq:Sintder} implies that for $s<r<t$,
\[S(t,s) - S(r,s) = -\int_r^t A(\tau) S(\tau,s)  \, d\tau\]
and thus $D_t S(t,s) = -A(t) S(t,s)$ in $\calL(X_0)$ for almost all $s<t$.

To prove \eqref{eq:propevol2} we use a similar duality argument as in \cite[Section 6]{AT3} and \cite[Proposition 2.9]{AFT}. Fix $t_0\in \R$. Clearly, $A(t_0-\tau)^*$ has symbol $A_{\#}(t_0-\tau,\xi)^*$ and hence generates a strongly continuous evolution family, $(W(t_0;\tau, s))_{s\leq \tau}$. Now as in \cite[Proposition 2.9]{AFT} one can deduce $S(t,s)^* = W(t;t-s, 0)$. Therefore, applying \eqref{eq:propevol1} to $W(t;t-s, 0)$, we see that for almost all $s<t$
\[D_sS(t,s)^* = D_sW(t;t-s,0) = A(t-(t-s))^* W(t;t-s,0) = A(s)^* S(t,s)^*, \]
and hence for almost all $s<t$,
\begin{align}\label{eq:Sstarid}
D_s S(t,s) = (A(s)^* S(t,s)^*)^* \ \text{on} \ \calL(X_0).
\end{align}
Now the result follows since the identity $(A(s)^* S(t,s)^*)^* = S(t,s) A(s)$ holds on $X_1$. In particular, we find that for $g\in X_1$,
\[S(t,s) g - S(t,s-\varepsilon) g = \int_s^t S(t,r) A(r) S(s,s-\varepsilon) g\, dr\]
and letting $\varepsilon\downarrow 0$, yields
\[S(t,s) g - g = \int_s^t S(t,r) A(r)  g\, dr\]
from which we obtain Definition \ref{def:evolut} (iv).

From the above construction and the properties of $W$ one sees that $D_s S(t,s)$ is locally integrable on $(-\infty, t)$, and that $s\mapsto S(t,s)\in \calL(X_0)$ is H\"older continuous on $(-\infty, -\varepsilon+t)$ for any $\varepsilon>0$. Combining this with the H\"older continuity of $t\mapsto S(t,s)$, we see that $(t,s)\mapsto S(t,s)\in \calL(X_0)$ is continuous on $\{(t,s):s<t\}$.
\end{proof}

\section{Proofs Theorems \ref{teo:mainxdep} and \ref{teo:mainxdepdiv}\label{sec:proofs}}

To prove Theorems \ref{teo:mainxdep} and \ref{teo:mainxdepdiv} we check the conditions of \cite[Theorem 4.9]{GV}.

\subsection{$\Rr$-boundedness of integral operators}
For details on $R$-boundedness  we refer to \cite{CPSW, DHP, KW}.

Let $\calK$ be the class of kernels $k\in L^1(\R)$ for which $|k|*f\leq Mf$ for all simple functions $f:\R\to \R_+$, where $M$ denotes the Hardy-Littlewood maximal operator.

Suppose $T:\{(t,s)\in \R^2: t\neq s\}\to \calL(X)$ is such that for all $x\in X$, $(t,s)\mapsto T(t,s)x$ is measurable. For $k\in \calK$ let
\begin{equation}\label{eq:IkTdefprelim}
I_{k T} f(t) = \int_{\R} k(t-s) T(t,s) f(s)\, ds.
\end{equation}
Consider the family of integral operators $\I:=\{I_{k T}: k\in \calK\}\subseteq\calL(L^{p}(\R;X))$.

A sufficient condition for $\Rr$-boundedness of such families was obtained in \cite[Theorem 1.1]{GLV} in the case $X = L^q(\Omega,w)$ in terms of a boundedness condition for $T(t,s)\in \calL(L^q(\Omega,w))$, where $\Omega\subseteq \R^d$ is open and $w$ is an $A_q$-weight. This result can be extended to the following setting.
\begin{proposition}\label{prop:weightedRextended}
Let $q_0\in (1, \infty)$, $w\in A_{q_0}$ and $H$ be a Hilbert space. Let $\{T(t,s):s,t\in \R\}$ be a family of bounded operators on $L^{q_0}(\R^{d},w;H)$. Assume that for all $A_{q_0}$-weights $w$,
\begin{equation}\label{eq:weightedcond}
\|T(t,s)\|_{\calL(L^{q_0}(\R^{d},w;H))}\leq C,
\end{equation}
where $C$ is $A_{q_0}$-consistent and independent of $t,s\in \R$. Then the family of integral operators
$\I = \{I_{k T}: k\in \calK\}\subseteq \calL(L^p(\R,v;L^q(\R^{d},w;H)))$ as defined in \eqref{eq:IkTdefprelim} is $\Rr$-bounded for all $p,q\in (1, \infty)$ and all $v\in A_p$ and $w\in A_q$. Moreover, in this case the $\Rr$-bounds $\Rr(\I)$ are $A_{p}$- and $A_q$-consistent.
\end{proposition}

In the case $H$ has finite dimension $N$, one could apply \cite[Theorem 1.1]{GLV} coordinate wise, but this only yields estimates with $N$ dependent constants. To avoid this, one can repeat the argument from \cite[Theorem 1.1]{GLV} almost literally. Only the definition of $\ell^s$-boundedness (see \cite{KunstUll}) has to be extended to the $H$-valued setting in the following way:

A family of operators $\T \subseteq \calL(X,Y)$ is said to be {\em $\ell^s_H$-bounded} if there exists a constant $C$ such that for all $N\in \N$, all sequences $(T_n)_{n=1}^N$ in $\T$ and $(x_n)_{n=1}^N$ in $X$,
\begin{align}\label{eq:defRbddsquarefun}
\Big\|\Big(\sum_{n=1}^N \|T_n x_n\|_{H}^{s}\Big)^{1/s}\Big\|_{Y} \leq C\Big\|\Big(\sum_{n=1}^N \|x_n\|_{H}^{s}\Big)^{1/s}\Big\|_{X},
\end{align}
where the usual modification has to be used if $s=\infty$. In the case $X=Y=L^{q}(\R^d;H)$, the $\ell^2_H$-boundedness is equivalent to $\Rr$-boundedness.

Now we can check the conditions of \cite[Theorem 4.9]{GV} in the case the coefficients of $A$ are $x$-independent.
\begin{proposition}\label{prop:evfamSys}
Let $q\in (1, \infty)$, $w\in A_q$ and set $X_0 = L^q(\R^d,w;\C^N)$ and $X_1 = W^{2m,q}(\R^d,w;\C^N)$.
Assume $A$ is of the form \eqref{eq:defoperatortimeind}. Let $\kappa, K>0$ be such that for all $t\in \R$, $A(t) \in\Ell(\kappa,K)$. Let $A_{0}:=\delta (-\Delta)^{m}I_{N}$ for $\delta\in (0,\kappa)$ fixed. Then the following properties hold:
\begin{enumerate}
\item\label{it:Hinfty} $A_0$ has a bounded $H^\infty$-calculus of any angle $\sigma\in (0,\pi/2)$.
\item\label{it:genTweight} $A(t) - A_0\in \Ell(\kappa-\delta, K+\delta)$ and generates a unique evolution family $T(t,s)$ with the property that
\[
\|T(t,s)\|_{\calL(L^{q}(\R^{d},w;\C^N))}\leq C, \ \ s\leq t,
\]
where $C$ is $A_q$-consistent.
\item\label{it:commute} $T(t,s)$ commutes with $e^{-rA_0}$ for all $s\leq t$ and $r\geq 0$.
\end{enumerate}
\end{proposition}
\begin{proof}
\eqref{it:Hinfty}: The symbol of $A_0$ is $\delta|\xi|^{2m} I$, where $I$ is the $N\times N$ identity matrix and the fact that the operator $A_0$ has a bounded $H^\infty$-calculus follows from the weighted version of the Mihlin multiplier theorem (see \cite[Example 10.2b]{KW} for the unweighted case).

\eqref{it:genTweight}: For $|\xi|=1$ and $x\in \C^N$,
\begin{align*}
\text{Re}(\lb x, (A_{\#}(\xi) - \delta|\xi|^{2m}) x\rb) \geq(\kappa-\delta) \|x\|^2.
\end{align*}
Also the coefficients of the symbol of $A_0$ are $\delta$ or $0$, so indeed $\Ell(\kappa-\delta, K+\delta)$ and the required result follows from Theorem \ref{teoMihCondSys}.

\eqref{it:commute} From the proof of Theorem \ref{teoMihCondSys} we see that $T(t,s)$ is given by a Fourier multiplier operator. Also $e^{-rA_0}$ is given by a Fourier multiplier with symbol $e^{-r |\xi|^{2m}}I_N$. This symbol clearly commutes with any matrix in $\C^{N\times N}$, and hence with the symbol of $T(t,s)$. Therefore, the operators $T(t,s)$ and $e^{-rA_0}$ commute.
\end{proof}

\begin{proof}[Proof of Theorem \ref{teo:mainxdep}]
{\em Step 1}: First assume $A$ is of the form \eqref{eq:defoperatortimeind}, i.e.\ it has $x$-independent coefficients. Then
by Propositions \ref{prop:weightedRextended} and \ref{prop:evfamSys}, the conditions of \cite[Theorem 4.9]{GV} are satisfied. Therefore, the existence and uniqueness result and \eqref{eq:teomainxdep} follow for any fixed $\lambda_0>0$ and the constant in \eqref{eq:teomainxdep} is $A_p$-$A_q$-consistent.

{\em Step 2}: In order to complete the proof, one can repeat the argument of \cite[Theorem 5.4]{GV} by replacing the scalar field by $\C^N$. Note that to apply the localization argument and to include the lower order terms, one has to use the interpolation estimate from Theorem \ref{teoHHHsys}.
\end{proof}

\begin{proof}[Proof of Theorem \ref{teo:mainxdepdiv}]
{\em Step 1}: First assume $A$ is of the form \eqref{eq:defoperatortimeind} again. Now we use the result from Theorem \ref{teo:mainxdep} in the $x$-independent case in a similar way as in \cite[Theorem 4.4.2]{krylov}. Let $\lambda\geq \lambda_0$, where $\lambda_0>0$ is fixed. For each $|\alpha|\leq m$, let
$v_{\alpha}\in W^{1,p}(\R,v;X_0) \cap L^p(\R,v;X_1)$ be the unique solution to
\[v'(t,x)+(\lambda+A(t))v(t,x)=f_{\alpha}(t,x),\ \ \ t\in\R,\ x\in\R^{d}.\]
Then by Theorems \ref{teoHHHsys} and \ref{teo:mainxdep}
\[\sum_{|\beta|\leq m} \lambda^{1-\frac{|\beta|}{2m}} \|D^{\beta + \alpha} v\|_{L^p(\R,v;X_0)} \leq C \lambda^{\frac{|\alpha|}{2m}} \|f_{\alpha}\|_{L^p(\R,v;X_0)}.\]
Therefore, setting $u = \sum_{|\alpha|\leq m} D^{\alpha} v_{\alpha}$ and using the fact $D^{\alpha}$ and $A$ commute in distributional sense, we find that $u$ is a weak solution to \eqref{prob:systemintrodiv} and
that \eqref{eq:teomainxdepdiv} holds. Uniqueness follows from \eqref{eq:teomainxdepdiv} as well.

{\em Step 2}: To obtain the result for general $A$, one can use a localization argument with weights and extrapolation as in \cite[Theorem 5.4]{GV} in the non-divergence form case. This argument works in the divergence form case as well (see \cite[Section 13.6]{krylov} for the elliptic setting).
\end{proof}

\subsection{Consequences for the initial value problem}

In this section we consider the initial value problem
\begin{equation}\label{eq:Cauchy}
\begin{aligned}
u'(t,x)+A(t)u(t,x)& =f(t,x),\ t\in (0,T), \ x\in \R^d,
\\ u(0,x)& =u_0(x),  \ \ x\in \R^d,
\end{aligned}
\end{equation}
where $A$ is in non-divergence form and satisfies the same condition (C) as in Theorem \ref{teo:mainxdep}.
A function $u:\R\times\R^d\to \C^N$ is called a {\em strong solution} of \eqref{eq:Cauchy} when
all the above derivatives (in the sense distributions) exist, \eqref{eq:Cauchy} holds almost everywhere
and for all bounded sets $Q\subseteq\R^d$, $u(t,\cdot)\to u_0$ in $L^1(Q;\C^N)$.

In order to make the next result more transparent we only consider power weights in the time variable.
Maximal regularity results with power weights are important in the study of nonlinear PDEs (see \cite{Grisvard, KPW, Lun, MS12a, MS12b, PS04} and references therein). For instance it allows one to work with a larger class of initial values.
\begin{theorem}
Let $T\in (0,\infty)$. Let $p,q\in (1,\infty)$, $\gamma\in [0, p-1)$, $v_{\gamma}(t) = t^{\gamma}$, $w\in A_{q}(\R^{d})$, $X_0 = L^q(\R^d,w;\C^N)$ and $X_1 = W^{2m,q}(\R^d,w;\C^N)$. Assume condition (C) holds and let $s = 2m\Big(1-\frac{1+\gamma}{p}\Big)$. Then for every $f\in L^{p}(0,T,v_{\gamma};X_{0})$ and every $u_0\in B^{s}_{q,p}(\R^d,w)$ there exists a unique strong solution $u\in W^{1,p}(0,T,v_{\gamma};X_0)\cap L^p(0,T,v_{\gamma};X_1)\cap C([0,T];B^{s}_{q,p}(\R^d,w))$ of \eqref{eq:Cauchy}. Moreover, there is a constant $C$ depending on $\gamma$, $w$, $p$, $q$, $d$, $m$, $\kappa$, $K$, $\omega$ and $T$ such that
%\begin{equation}\label{eq:teomainxdep}
\[\begin{aligned}
\|u\|_{L^p(0,T,v_{\gamma};X_1)}+& \|u\|_{W^{1,p}(0,T,v_{\gamma};X_0)} + \|u\|_{C([0,T];B^{s}_{q,p}(\R^d,w))}\\ & \leq C\|f\|_{L^{p}(0,T,v_{\gamma};X_{0})} + C\|u_0\|_{B^{s}_{q,p}(\R^d,w)}.
\end{aligned}\]
%\end{equation}
\end{theorem}

\begin{proof}
Substituting $v(t,\cdot) = e^{-\lambda t} u(t,\cdot)$ it follows that we may replace $A$ by $\lambda+A$ for an arbitrary $\lambda$. Therefore, extending $f$ as zero outside $(0,T)$, by Theorem \ref{teo:mainxdep} we may assume that $A$ has maximal $L^p_v$-regularity as defined in \cite[Definition 4.11]{GV} for any $v\in A_p$.  Recall from \cite[Example 9.1.7]{GrafakosModern} that $v_{\gamma}\in A_p$.

By \cite{PS04} and the maximal $L^p$-regularity estimate from Theorem \ref{teo:mainxdep} (also see \cite[Section 4.4]{GV}), we need that $u_0\in (X_0, X_1)_{1-\frac{1+\gamma}{p},p}$ to obtain the well-posedness result and the estimate. The latter real interpolation space can be identified with $B^{s}_{q,p}(\R^d,w)$. Indeed, in the case $w=1$, this follows from \cite[Theorem 6.2.4]{BergLof} or \cite[Remark 2.4.2.4]{Tr1}. In the weighted setting this follows from the inhomogeneous case of \cite[Theorem 3.5]{Bui82}.
\end{proof}

\def\cprime{$'$}

\end{document}